\documentclass[12pt,reqno]{amsart}
\usepackage[margin=1in]{geometry}
\usepackage[utf8]{inputenc}
\usepackage{setspace}
\usepackage{hyperref}
\usepackage{enumitem}
\usepackage{graphicx}
\usepackage{xcolor}
\usepackage[bitstream-charter]{mathdesign} 
\usepackage{amsmath, mathtools}
\usepackage{amsthm}
\usepackage{enumitem}
\usepackage{caption}
\usepackage{subcaption}
\usepackage{algorithm2e}

\usepackage{todonotes}
\usepackage{tikz}
\usetikzlibrary{shapes}
\usepackage[normalem]{ulem}
\newtheorem{theorem}{Theorem}
\newtheorem{prop}{Proposition}
\newtheorem{observation}{Observation}

\newtheorem{corollary}{Corollary}

\newtheorem{conjecture}{Conjecture}

\renewcommand{\d}{\delta}
\newcommand{\D}{\Delta}

\newcommand{\rbrac}[1]{\left(#1\right)} 
\newcommand{\cbrac}[1]{\left\{ #1\right\}} 
\newcommand{\abrac}[1]{\left| #1\right|} 
\newcommand{\ceilbrac}[1]{\left\lceil #1\right\rceil} 
\newcommand{\mc}[1]{\mathcal{#1}}

\newcommand{\eps}{\varepsilon}

\def\g{\gamma}

\title{Some progress on $t$-tone coloring}
\author{Patrick Bennett}
\address{Department of Mathematics, Western Michigan University, Kalamazoo, MI, USA}
\thanks{The first author was supported in part by Simons Foundation Grant \#848648.}
\email{\tt patrick.bennett@wmich.edu}

\author{Jade Nichols}
\address{Department of Mathematics, Western Michigan University, Kalamazoo, MI, USA}
\thanks{The second author was supported in part by funding from the Western Michigan University College of Arts and Sciences.}
\email{\tt nicholsjade26@gmail.com}

\begin{document}

\maketitle

\begin{abstract}
    A $t$-tone coloring of a graph $G$ assigns to each vertex a set of $t$ colors such that any pair of vertices $u, v$ with distance $d$ can share at most $d-1$ colors. In this note, we prove several new results on $t$-tone coloring. For example we prove a new result for trees of large maximum degree, as well as some results for the cartesian power of a graph. We also make a conjecture about trees.
\end{abstract}
\section{Introduction} 

Let $G$ be a simple graph, i.e.~$G$ has no loops or multiple edges (indeed, in this paper all graphs are simple). Let $t \ge 1$ be an integer and let $C$ be a set of colors. We say that a labeling of the vertices of $G$ with sets of $t$ colors $\ell :V(G) \rightarrow \binom{C}{t}$ is a {\em $t$-tone coloring} if $|f(u) \cap f(v)| < d(u, v)$ for all $u, v \in V(G)$. In other words, a pair of vertices at distance $d$ share at most $d-1$ colors. Note that this requirement is only nontrivial for $d=1, \ldots, t$. Two vertices in different components of $G$ have distance $\infty$ and are allowed to share any number of colors.  \\

We let $\tau_t(G)$, the {\em $t$-tone chromatic number of $G$}, be the smallest number of colors $|C|$ possible in a $t$-tone coloring of $G$. Note when $t=1$ this is equivalent to the standard chromatic number $\chi(G)$. Also, if $G$ is disconnected then $\tau_t(G)$ is just the maximum $t$-tone chromatic number of the components of $G$. For the rest of this note we will focus on connected graphs.\\

The $t$-tone chromatic number of a graph was first suggested by Gary Chartrand and initially explored by Fonger, Goss, Phillips and Segroves \cite{FGPS} under the supervision of Ping Zhang. Since their introduction they have had further investigation from several groups, see eg. \cite{BBDF,  TrTn, CKK13, CH, LMM}. Notably, in \cite{CKK13} came the first asymptotic results dealing with large graphs 
, which was furthered by \cite{BBDF} which focused purely on large graphs, namely the random graph.\\

In this note we will make some progress on the $t$-tone coloring problems with focus on families of large graphs. We have several new results, and the rest of the introduction is organized into subsections to discuss results of different types. 

\subsection{Trees }
First we investigate $t$-tone coloring for trees. Trees are of course bipartite, and so $\tau_1(T)=2$ for any nontrivial tree $T$. $\tau_t(T)$ for $t\ge 2$ becomes more interesting. Fonger, Goss, Phillips and Segroves \cite{FGPS} found an explicit formula for the $2$-tone chromatic number of a tree depending only on its maximum degree. 
\begin{theorem}[\cite{FGPS}]\label{thm:2-tonetree} If $T$ is a tree with $\Delta(T)=\Delta$ then 
    \[
    \tau_2(T) = \left \lceil \frac{\sqrt{8 \D +1}+5}{2} \right \rceil.
    \]
\end{theorem}
Thus if $\D$ is large, $\tau_2(T)$ is approximately $\sqrt{2\D}$. The lower bound on $\tau_2(T)$ in Theorem \ref{thm:2-tonetree} follows from a more general observation.
\begin{observation}\label{obs:DeltaLB}
    Let $G$ be a graph with $\Delta(G) = \D$. Then for any $t \ge 2$ we have
    \[
    \tau_t(G) \ge \left\lceil\frac{2t+1+\sqrt{1+4t(t-1)\Delta}}{2}\right\rceil > \sqrt{t(t-1)\D}.
    \]
\end{observation}
\begin{proof}
    Let $v$ be a vertex of degree $\D$ in $G$, and suppose we have a $t$-tone coloring of $G$ using color set $C$. Then $v$ gets $t$ colors, and the neighbors of $v$ cannot use those colors at all. Furthermore no two neighbors of $v$ can share two colors, so there are $\Delta \binom t2$ distinct pairs of colors used by the neighbors of $v$. Thus we have $\Delta \binom t2 \le \binom{|C|-t}{2}$. Solving for $|C|$ yields the first inequality, and the second is elementary.
\end{proof}

The proof above only looks at pairs of vertices whose distance is 1 or 2, even though a $t$-tone coloring has requirements for pairs at distance up to $t$. So it would seem likely that Observation \ref{obs:DeltaLB} is not very good for $t \ge 3$. Surprisingly, Cranston, Kim and Kinnersley \cite{CKK13} proved that to the contrary we have $\tau_t(T)= O(\sqrt{\D})$ for all $t$.
\begin{theorem}[ \cite{CKK13}]\label{thm:CKKtrees}
    For each $t \ge 2$ there is a constant $c=c(t)$ such that for all trees $T$ with $\Delta(T)=\Delta$ we have
    \[
    \tau_t(T) \le c\sqrt{\D}.
    \]
\end{theorem}

In this note we asymptotically determine $\tau_t(T)$ for fixed $t$ and large $\D(T)$. In particular, we show that Observation \ref{obs:DeltaLB} is asymptotically optimal for large $\D$.
\begin{theorem}\label{thm:trees}
    For any integer $t \ge 2$ and real number $\delta >0$ there exists a $\D_0=\D_0(t, \d)$ such that for all trees $T$ with maximum degree $\D(T) = \D \ge \D_0$ we have 
    \begin{equation}\label{eqn:trees}
      \tau_t(T)\le (1  + \delta) \sqrt{t(t-1)\D}.
    \end{equation}
\end{theorem}
\noindent We conjecture that $\tau_t(T)$ can be determined precisely by $\Delta(T)$ when it is large. For $k \ge 2$ we let $S_k$ be the star with $k$ leaves, and we let $T_k$ be the infinite $k$-regular tree. Note that if $T$ is a tree with $\D = \D(T)$ then we have $S_\D \subseteq T \subseteq T_\D$ and so $\tau_t(S_\D) \le \tau_t(T) \le \tau_t(T_\D)$ for any $t$. We conjecture that equality holds for $\D$ large with respect to $t$. 
\begin{conjecture}\label{conj:tree}
    For every $t$ there exist a $\Delta_0=\D_0(t)$ such that $\tau_t(S_\D) = \tau_t(T_\D)$ for all $\D \ge \D_0$.
\end{conjecture}
\noindent The conjecture, if true, would imply that $\tau_t(T)$ is completely determined by $\D(T)$ when $\D(T)$ is large with respect to $t$. \\

\subsection{Sparse random graphs}
We let $G_{n,p}$ be the classical Erd\H{o}s-R\'enyi-Gilbert random graph, i.e.~the graph on $n$ vertices where each edge is present with probability $p$ (independently from all other edges). If the probability of an event $\mathcal{E}_n$ tends to $1$ as $n \rightarrow \infty$, we say the event occurs \textit{with high probability} (w.h.p.). In \cite{BBDF} it was shown by Bal, the first author, Dudek, and Frieze that

\begin{theorem}[\cite{BBDF}]\label{thm:oRG}
For any fixed integer $t\ge 2$ and real $c>0$  there exist constants $c_1$ and $c_2$ such that for $p=\frac cn$ we have with that w.h.p.
\[
c_1\sqrt{\Delta(G)}\leq\tau_t(G_{n,p})\leq c_2\sqrt{\Delta(G)}
\]
\end{theorem}

\begin{theorem}\label{thm:RG}
For any fixed integer $t\ge 2$ and real $c>0$ we have that for $p=\frac cn$ the following holds w.h.p. as $n \rightarrow \infty$.
    \[\tau_t(G_{n,p})=\left(1+o(1)\right)\sqrt{t(t-1)\D(G(n, p))} 
    \]
\end{theorem}
Of course, it is well known (see e.g.~\cite{friezebook}) that for $p=c/n$ we have w.h.p. $\Delta(G(n, p)) = (1+o(1)) \log n / \log \log n$. Thus we have
\begin{corollary}
 In the setting of Theorem \ref{thm:RG} we have w.h.p.
    \[    \tau_t(G_{n,p})=(1+o(1))\sqrt{\frac{t(t-1)\log(n)}{\log\log(n)}}.
    \]
\end{corollary}
In this paper we use standard asymptotic notation $O(-), o(-), \Omega(-), \Theta(-)$. We will specify what variable the asymptotics are in (for example, in this subsection it is as $n \rightarrow \infty$).

\subsection{Complete multipartite graphs}

For $a_i\in\mathbb{N}$ for $i\in[n]$ we define $K_{a_1,a_2,\ldots,a_n}$ to be the complete $n$-partite graph with parts of order $a_1,\ldots, a_{n}$. The following precise result is known for the $2$-tone chromatic number proved by Loe, Middelbrooks,and Morris in \cite{LMM}. 
\begin{theorem}[\cite{LMM}]
    For all $a_i\in \mathbb{N}$ over $i\in[b]$ we have that

\[\tau_2(K_{a_1,a_2,\ldots,a_b})=\sum_{i=1}^b \left\lceil\frac{1+\sqrt{8a_i+1}}{2}\right\rceil\]
\end{theorem}

 It was by this result we were motivated to investigate the general case for $t$-tone in our following result, which is basically a corollary of Theorem \ref{thm:trees}. This next result can be viewed as an approximate extension of the above theorem to all $t$ for complete multipartite graphs where all parts are large.
\begin{theorem}\label{thm:multipartite}
    
    For all positive integers $t$ and $b$ and $\delta>0$ there exists an $a_0$ such that for $a_1,a_2,\ldots,a_b$ where $a_i>a_0$ for $i\in[b]$ we have that
    \[\sum_{i=1}^b\sqrt{t(t-1)a_i}\leq\tau_t(K_{a_1,a_2,\ldots,a_b})\leq(1+\delta)\sum_{i=1}^b\sqrt{t(t-1)a_i}\]
\end{theorem}

\subsection{2-tone coloring graphs with small chromatic number}

Cranston, Kim and Kinnersley \cite{CKK13} proved the following.
\begin{theorem}[ \cite{CKK13}]\label{thm:CKKDelta}
    Let $G$ be a graph with $\Delta(G)=\D$. Then 
    \[
    \tau_2(G) \le \lceil (2 + \sqrt 2) \D \rceil, \qquad \text{ and if $G$ is bipartite then }\quad \tau_2(G) \le 2 \lceil \sqrt 2 \D \rceil.
    \]
   
\end{theorem}
Our next result is partly inspired by the above result for bipartite $G$. The $2$-tone coloring they used in the proof had two disjoint sets of colors, one for each part of the bipartition of $G$. Thus, pairs of vertices $u, v$ at distance 1 would never share any colors, and so one only has to address the pairs at distance 2. For our result we will let $\chi(G)=k$, and in the proof we will use $k$ disjoint sets of colors, one for each color class in a proper $k$-coloring of $G$.\\

\begin{theorem}\label{thm:smallchi1}
    Let $G$ be a graph with $\chi(G)=k$. Consider a proper $k$-coloring of $G$ and let $V_1, \ldots, V_k$ be the color classes. Let $G_i$ be the graph with vertex set $V_i$ such that for any two vertices $u, v \in V_i$ we have $uv \in E(G_i)$ whenever $d_G(u, v) =2$. Then 
    \begin{equation}\label{eqn:smallchi}
        \tau_2(G) \le k+ \sum_{i=1}^k \ceilbrac{\sqrt{2\chi(G_i)}}.
    \end{equation}
\end{theorem}

\begin{theorem}\label{thm:smallchi2}
    There is an absolute constant $C$ such that we have the following. Let $G$ be a graph with maximum degree $\Delta$. Suppose $G$ has the following property: for every vertex $u$ there are at most $\Delta^4/f$ pairs of vertices $v, w$ such that $d(u, v) = d(u, w) = d(v, w)=2$. Then we have 
    \begin{equation}\label{eqn:smallchicor}
        \tau_2(G) \le  \chi(G) \rbrac{1+ \frac{C\Delta}{\sqrt{\log f}}} .
    \end{equation}
\end{theorem}

\subsection{Cartesian powers}

Fonger, Goss, Phillips and Segroves \cite{FGPS} studied $\tau_2(Q_n)$, where $Q_n = K_2^n$ is the Cartesian product $K_2 \times \cdots \times K_2$. For $n\le 3$ and proved that $7 \le \tau_2(Q_4) \le 8$. For all $n$ they proved that $\tau_2(Q_n) \le 2^{n-1}+2.$  Theorem \ref{thm:CKKDelta} gives $\tau_2(Q_n) \le 2 \lceil \sqrt 2 n \rceil$, which is a vast improvement for large $n$. Our next result provides a further improvement for large $n$, and generalizes $Q_n$ to any Cartesian power of a fixed graph.

\begin{theorem}\label{thm:cube}
    There is an absolute constant $C$ such that we have the following. Fix a graph $G$ with chromatic number $\chi(G)=k$. Let $G^b$ be the Cartesian product $G \times G \times \dots \times G$ of $b$ copies of $G$. Then  
    \[
    \tau_2(G^b) \le \frac{Ckb}{ \sqrt{\log b}}
    \]
\end{theorem}

Our next theorem is about $t$-tone coloring the Cartesian power of a clique. 

\begin{theorem}\label{thm:cliqueproduct}
   Fix integers $b, t \ge 1$ and a real number $\d>0$. For all sufficiently large $n$ (depending on $b, t, \d$)  we have that
    \begin{equation}\label{eqn:cliqueproduct}
        \tau_t(K_n^b)\leq (1+\d)tn.
    \end{equation}
\end{theorem}
Note that if $n = \max\{n_1, \ldots, n_b\}$ then we have
\[
K_n \subseteq K_{n_1} \times \cdots \times K_{n_b} \subseteq K_n^b
\]
and so Theorem \ref{thm:cliqueproduct} implies that for large $n$ we have
\[tn = \tau_t(K_n) \le \tau_t(K_{n_1} \times \cdots \times K_{n_b}) \le \tau_t(K_n^b) \le (1+\delta)tn.\]

 A Latin square is an $n\times n$ matrix such that every entry is an element from $[n]$ and every $i\in[n]$ appears in each row and column exactly once. Consider two Latin squares $A$ and $B$ where entry $(i,j)$ is $a_{i,j}$ and $b_{i,j}$ for $A$ and $B$ respectively. We call these Latin squares \textit{ orthogonal} if 
for every $(i,j),(i',j')\in[n]^2$ we have that $(a_{i,j},b_{i,j})=(a_{i',j'},b_{i',j'})$ if and only if $(i, j) = (i', j')$. A collection of Latin squares is called \textit{mutually orthogonal} if all pairs are orthogonal. We use the initialism MOLS for mutually orthogonal Latin squares. Bickle \cite{Bickle1} used orthogonal Latin squares to prove that $\tau_2(K_n^2) = 2n$ for all $n \ge 3.$ Our next result uses similar ideas to obtain $t$-tone colorings of $K_n^2$ for general $t$.\\

\begin{theorem} \label{TOLS}
    If there exists a collection of $n \times n$ MOLS of cardinality $t$  then $\tau_t(K_n^2)=tn$.
\end{theorem}

Latin squares are very well studied, and there are some known results on MOLS. For example, we can use some old results of Beth \cite{MOLS} and Macneish \cite{McNT} to obtain the following corollary of Theorem \ref{TOLS}.

\begin{corollary}\label{cor:MOLS}
    We have $\tau_t(K_n^2)=tn$ in either of the following two cases:
    \begin{enumerate}[label=(\alph*)]
        \item\label{case:beth} $t \le n^{5/74}$,
        \item\label{case:macneish} $n$ is a prime power and $t \le n-1$.
    \end{enumerate}
\end{corollary}

\subsection{$t$-tone coloring a fixed graph $G$ for large $t$}

Much of the previous work on $t$-tone coloring is for fixed $t$ and considering various graphs $G$. Our first result considers the ``opposite'' situation, where we fix $G$ and consider large values of $t$. Note that for any $t, G$ we always have $\tau_t(G) \le t|V(G)|$ since we can label all vertices with disjoint sets of colors. We will now see that for a fixed connected graph $G$ and large $t$ this upper bound is almost correct since $\tau_t(G) = tn -\kappa_G$ where $\kappa_G$ depends only on $G$, and in fact we determine this $\kappa_G$ precisely.
\begin{theorem}\label{thm:larget}
    Let $G$ be any connected graph with $n$ vertices. Then 
    \begin{equation}\label{eqn:larget}
        \tau_t(G)\ge tn -\sum_{u,v\in V(G)} \Big(d(u,v)-1 \Big),
    \end{equation}
    where the sum is over all (unordered) pairs of distinct vertices $u, v$.
    Furthermore if $G$ has diameter $D$ and $t\geq (n-1)(D-1)$ then we have equality in \eqref{eqn:larget}.
\end{theorem}

\subsection{Organization of paper}
In Section \ref{sec:tools} we discuss some terminology and tools we plan to use. In Section \ref{sec:trees} we prove Theorem \ref{thm:trees} about trees, Theorem \ref{thm:RG} about sparse random graphs, and Theorem \ref{thm:multipartite} about complete multipartite graphs. In Section \ref{sec:smallchi} we prove Theorems \ref{thm:smallchi1} and \ref{thm:smallchi2} bounding $\tau_2(G)$ in terms of $\chi(G)$. In Section \ref{sec:powers} we prove Theorems \ref{thm:cube}, \ref{thm:cliqueproduct} and \ref{TOLS} about Cartesian powers of graphs, as well as Corollary \ref{cor:MOLS}. In Section \ref{sec:larget} we prove Theorem \ref{thm:larget} about $\tau_t(G)$ where $G$ is fixed and $t$ is large. In Section \ref{sec:treeconj} we give some evidence for Conjecture \ref{conj:tree} by working out some small cases. Concluding remarks are in Section \ref{sec:conclusion}.

\section{Terminology and tools}\label{sec:tools}

In this section we introduce some terminology and tools.\\ 

Let $\mc{H}$ be a hypergraph. If all the edges of $\mc{H}$ have size $k$ we say $\mc{H}$ is \emph{$k$-uniform}, and if all edges have size at most $k$ we say $\mc{H}$ is \emph{$k$-bounded}. A \emph{matching} in $\mc{H}$ is a set of pairwise disjoint edges. We denote by $\mc{H}^{(j)}$ the hypergraph on the same vertex set as $\mc{H}$ and with all the edges of size $j$.
For a hypergraph $\mc{H}$ and a vertex $v \in V(\mc{H})$, 
we denote by $d_\mc{H}(v)$ the \emph{degree} of $v$; 
that is, the number of edges containing $v$.
We use $\Delta(\mc{H})=\max_{u\in V(\mc{H})}d_\mc{H}(u)$ to denote the \emph{maximum degree} of $\mc{H}$.
Similarly, we define the \emph{minimum degree} $\delta(\mc{H})$ of $\mc{H}$.
For $j\geq 2$, we denote by $\Delta_j(\mc{H})$ the maximum number of edges that contain a particular set of $j$ vertices.\\

For a hypergraph $\mc{H}$ and a partition $V(\mc{H}) = X \cup Y$, we say that $\mc{H}$ is \emph{bipartite} and has \emph{bipartition} $X \cup Y$ if every edge of $\mc{H}$ contains exactly one vertex of $X$. In this case we say that a matching $\mc{M}$ in $\mc{H}$ is $X$-perfect if every vertex of $X$ is contained in some edge in $\mc{M}$. 

\begin{theorem}[{Delcourt and Postle \cite[Theorem~1.16]{DP22}}]\label{thm:DP}
For all  $k, \ell \ge 2$ and real $\eps \in  (0, 1)$, there exist an integer $d_0 > 0$ and real $\alpha > 0$ such that following holds for all $d \ge d_0$. Let $\mc{H}$ be a bipartite $k$-bounded hypergraph  with bipartition $X \cup Y$ such that
\begin{enumerate}[label=(H\arabic*)]
\item every vertex in $X$ has degree at least $(1+d^{-\alpha})d$ and every vertex in $Y$ has degree at most $d$, and \label{cond:h1}
\item $\Delta_2(\mc{H}) \le d^{1-\eps}$.\label{cond:h2}
\end{enumerate}

Then there exists an $X$-perfect 
matching of $\mc{H}$.
\end{theorem}

There are several results we could use in place of Theorem \ref{thm:DP}, and our choice to use this version is made out of convenience and familiarity. For our purposes, it suffices to know that ``nice'' hypergraphs have ``large'' matchings. Various versions of such a statement have been known since the 1980's starting with Pippinger and Spencer \cite{PS}. The work of Delcourt and Postle \cite{DP22} is actually much more general and has the power to find a matching which satisfies additional constraints which we will not use here. Mostly, we chose to use their result because it is convenient for us to obtain $X$-perfect matchings.\\

We will also use the following to prove Theorem \ref{thm:smallchi2}.
\begin{theorem}[Alon, Krivelevich and Sudakov \cite{AKS99}]\label{thm:AKS}
    There exists an absolute constant $c$ such that the following holds. Suppose $G$ is a graph with maximum degree $d$ such that for every vertex $v$, its neighborhood $N(v)$ contains at most $d^2/f$ edges. Then $\chi(G) \le cd/ \log f$.
\end{theorem}

\section{Trees, sparse random graphs, and complete multipartite graphs}\label{sec:trees}

In this section we prove we prove Theorems \ref{thm:trees}, \ref{thm:RG} and \ref{thm:multipartite}.

\begin{proof}[Proof of Theorem \ref{thm:trees}]

    Fix $t \ge 2, \delta>0$. Now for sufficiently large $\D$, we will use Theorem \ref{thm:DP} to generate a $t$-tone  coloring of the infinite $\D$-regular tree $T_\D$ using $(1  + \delta) \sqrt{t(t-1)\D}$ colors. Choose a vertex $r\in T_\D$ and let $P_i$ be the set of vertices at distance $i$ from $r$, partitioning the vertex set into $P_0 \cup P_1 \cup \cdots.$ We will give $r$ (the only vertex in $P_0$) an arbitrary set of $t$ colors, and proceed to inductively color the vertices in each $P_i$.\\

    Suppose $i \ge 1$ and we have already colored the vertices in $P_0 \cup \ldots \cup P_{i-1}$ in a valid way (i.e. we have a $t$-tone coloring of $T_\D[P_0 \cup \ldots \cup P_{i-1}]$). We define a hypergraph $\mc{H}=(X\cup Y, E)$ where 
    \[
    X=P_i, \qquad Y=\bigcup_{j=1}^{\lfloor t/2 \rfloor}P_{i-j}\times\binom{C}{2j}.
    \]
    Now we describe the edge set $E$. Observe that a vertex $v\in P_i$ has a unique path back to $r$, and denote this path by $p_v$. This path includes exactly one vertex from each $P_k$ for $k<i$. For each vertex $v\in P_i$ and  set $S$ of $t$ colors, $\mc{H}$ will have an edge $e_{v, S}$. We will interpret this edge as assigning $v$ the color set $S$. We make the following exception: $\mc{H}$ will not have the edge $e_{v, S}$ if assigning $S$ to $v$ would create a violation of the $t$-tone coloring property (i.e. if there is some vertex $w \in P_0 \cup \ldots \cup P_{i-1}$ at distance say $j$ from $v$ such that $w$ shares $j$ colors with $S$). Otherwise we will set 
    \[
    e_{v, S}:=\{v\}\cup\cbrac{\{u\}\times\binom{S}{2d(u,v)}: u\in p_v}.
    \]

    We claim an $X$-perfect matching of $\mc{H}$ provides a valid coloring of $P_i$. In other words, it extends our previous $t$-tone coloring of $T_\D[P_0 \cup \ldots \cup P_{i-1}]$ to a $t$-tone coloring of $T_\D[P_0 \cup \ldots \cup P_{i}]$. Be definition, we do not include any edge in $\mc{H}$ that would cause us to color a vertex in $P_i$ that would conflict with a vertex in $P_0 \cup \ldots \cup P_{i-1}$. So if a conflict occurs it will occur between two vertices in $P_i$. Consider $u,v\in P_i$. There is a unique $uv$-path, and say that $j$ is the smallest index such that this $uv$-path contains a vertex  $w\in P_j$. This implies that $w \in p_u \cap p_v$, and that $d(u, v) = 2(i-j)$. If our color assignment creates a conflict between $u$ and $v$, then our matching in $\mc{H}$ has two edges $e_{u, S_u}, e_{v, S_v}$ such that there is a set of $2(i-j)$ colors $S' \subseteq S_u \cap S_v$ of size $2(i-j)$. But then $(w,S') \in e_{u, S_u}\cap e_{v, S_v}$, contradicting that we have a matching. Thus it suffices to show that $\mc{H}$ has an $X$-perfect matching. So we move on to check that $\mc{H}$ satisfies the conditions of Theorem \ref{thm:DP}. We let
    \[
    \eps=\frac 1t, \qquad d=(1+\delta)^{t-2}\left(\sqrt{t(t-1)}\right)^{t}\frac{\D^{t/2}}{t!}
    \]

We check Condition \ref{cond:h1}. Given a vertex $v \in X$, to complete an edge we choose a set $S$ of $t$ colors, but we must only count those sets $S$ that would not conflict with the previously colored vertices in $P_0 \cup \ldots \cup P_{i-1}$. We observe that for a vertex $v\in P_i$ for $u\in P_j$ such that $j<i$ we have that if $d(u,v)=k$ then every $\binom{\sigma(u)}{k}$ set is forbidden from between them and there are $\binom{|C|}{t-k}$ ways to complete the forbidden coloring acknowledging that we are double counting. We now wish to count the number of colored vertices of distance $j\le t$. We may go up $j$ times and find one vertex there, we can then go up $j-1$ times then down giving us $\D-1$ vertices and in general we can go up $j-k$ times then down $k$ times to see $(\D-1)^k$ when $k\le\lceil\frac{j}{2}\rceil-1$ as if $k=\frac{j}{2}$ we reach layer $P_i$ and otherwise do do not. Thus if $B$ is the set of colorings we don't allow for a fixed vertex $v$ we have that
\begin{align*}
|B|&\le\sum_{j=1}^{t}\sum_{k=0}^{\lceil j/2\rceil-1}(\D-1)^k\binom{C}{t-j}\binom{t}{j}\\
&\leq \sum_{j=1}^{t}\sum_{k=0}^{\lceil j/2\rceil-1} \D^k((1+\d)^2t(t-1)\D)^{(t-j)/2}\\
&\leq \sum_{j=1}^tO(\D^{\lceil j/2\rceil-1}\D^{(t-j)/2})\\
&= O(\D^{\frac{t-1}2})=o(d)
\end{align*}

The inequality on the first and second line are immediate. The inequality on the third line comes from each term being of a lesser power tending to infinity for $k<\lceil j/2\rceil$. The final line is immediate.
This means we can bound $\d(X)$ with
\begin{align}
    \binom{C}{t}-|B|&\geq \frac{(1+\d)^t\left[\sqrt{t(t-1)\D}-t\right]^{t}}{t!}-o(d) \\
    & \geq \frac{(1+\d)^{t-1}\left[\sqrt{t(t-1)\D}-t\right]^{t}}{t!}\\
    &=(1+\d)d\geq(1+d^{-\eps})d
\end{align}

The first line is easy.
The second line holds, as we remove $o(d)$ from the term while multiplying to change the constant factor to the term of higher order. The third line is immediate.\\

Now consider $y=(u,S)\in Y$ where $u\in P_j$ and $|S|=2(i-j)$. There are at most $\D^{i-j}$ possible vertices in $P_i$ that $y$ could be in an edge with, and at most $\binom{C}{t-2(i-j)}$ ways to complete a set of $t$ colors containing $S$. Thus the degree of $y$ is at most (the inequalities will be explained)

\begin{align*}
    \D^{i-j}\binom{C}{t-2(i-j)} &\le \D^{i-j}\frac{[(1+\delta)\sqrt{t(t-1)\Delta}]^{t-2(i-j)}}{[t-2(i-j)]!}\\
    & \le (1+\delta)^{t-2} \Delta^{t/2}\frac{[\sqrt{t(t-1)}]^{t-2(i-j)}}{[t-2(i-j)]!}\\
    & \le (1+\delta)^{t-2} \Delta^{t/2} \frac{[\sqrt{t(t-1)}]^{t}}{t!} = d.
\end{align*}
The inequality on the first line is easy, and the second line follows from simplifying the first line and using $i-j \ge 1$. The third line follows from multiplying the second line by \newline $\sqrt{t(t-1)}^{2(i-j)} / (t)_{2(i-j)} \ge 1$.\\

We now check Condition \ref{cond:h2}. We fix two vertices in $\mc{H}$ and bound their codegree. We consider cases based on whether our fixed vertices are in $X$ or $Y$. Of course, two vertices in $X$ cannot be in any edge together in a bipartite hypergraph. Next, 
 consider the codegree of a vertex $x \in X$ and $y \in Y$.
Our vertex in $x$ tells us the vertex in $P_i$ which we are coloring, and our vertex $y$ determines at least $2$ colors, leaving at most $t-2$ colors to be chosen. Thus, the codegree of $x$ and $y$ is at most 
\[
O\rbrac{|C|^{t-2}} = O\rbrac{\Delta^{\frac t2 - 1}} = O\rbrac{d^{1- \frac2t}} < d^{1-\eps}.
\]
\\
Finally, consider two vertices in $Y$, say $y_1=(u_1,S_1)$ and $y_2=(u_2,S_2)$. Let $|S_1|=2j_1$, $|S_2|=2j_2$. Note that if $j_1 = j_2$ then either $u_1$ and $u_2$ do not have any common descendants in $P_i$ so the codegree of $y_1, y_2$ is 0 or $u_1=u_2$. In the event that $u_1=u_2$ we have that $S_1, S_2$ must be distinct color sets, else $y_1, y_2$ are the same vertex. Thus there exists a color in $S_1$ not in $S_2$ and vice versa, where each contains $2j_1$ colors. Thus they together must have at least $2j_1+2$ colors. As there are at most $\D^j$ vertices $x\in P_i$ such that $u$ lies in $p_x$ and at most $t-2j-2$ colors we have left to choose we have that our codegree is at most 
\[
O(\D^{j}|C|^{t-2j-2})=O(\D^{\frac t2-1})=O(d^{1-\frac 2t})<d^{1-\eps}.
\]

So assume without loss of generality that $j_1 < j_2$. To determine an edge containing $y_1$ and $y_2$ we can first choose a descendant of $y_1$ in $P_i$ in at most $\Delta^{j_1}$ ways, and then choose a set of $t$ colors containing $S_2$ in at most $|C|^{t-2j_2}$ ways. Thus, the codegree of $y_1$ and $y_2$ is at most
\[
\Delta^{j_1} |C|^{t-2j_2} = O\rbrac{\Delta^{j_1 + \frac t2 - j_2}}= O\rbrac{\Delta^{\frac t2 - 1}} = O\rbrac{d^{1- \frac2t}} < d^{1-\eps}.
\]

Thus, Theorem \ref{thm:DP} gives us an $X$-perfect matching, which tells us how to color $P_i$. By induction we are able to color the entire tree $T_\D$, completing the proof of Theorem \ref{thm:trees}.
\end{proof}

Next we will prove Theorem \ref{thm:RG}. First we discuss the proof of Theorem \ref{thm:oRG} by Bal, the first author, Dudek and Frieze. Our proof of Theorem \ref{thm:RG} will use some facts that were established in their proof. \\

Let $G:=G(n, p)$ where $p=c/n$ for some constant $c>0$. Let $H$ be the set of vertices with degree at least $\log^{1/4}n$, along with the vertices that are distance at most $2t-2$ away from such a vertex. In \cite{BBDF} they show that w.h.p.~$G[H]$ is a forest. They proceed to use Theorem \ref{thm:CKKtrees} to obtain a $t$-tone coloring to $H$. They proceed to show that one can greedily color all other vertices (after first uncoloring some vertices of $H$ to be recolored greedily) without introducing any new colors. Basically, we are forced to use so many colors near the high-degree vertices in $H$ that coloring the lower-degree vertices becomes easy. It follows that $\tau_t(G) = \tau_t(G[H]) = \tau_t(T)$ where $T$ is the component of $G[H])$ maximizing $\tau_t(T)$. Of course, $\log^{1/4}(n)\le \Delta(T) \le \Delta(G)$.\\

\begin{proof}[Proof of Theorem \ref{thm:RG}]
    The lower bound follows from Observation \ref{obs:DeltaLB}. We proceed to the upper bound.
    By the above discussion, we have that w.h.p.~there exists a tree $T$ such that         
    \[
    \log^{1/4}(n)\leq\D(T)\leq\D(G) \qquad \text{and} \qquad 
    \tau_t(G)=\tau_t(T)
    \]
    \\
   Thus $\D(T)\rightarrow \infty$ as $n \rightarrow \infty$, so we have that by Theorem \ref{thm:trees} 
    \[
    \tau_t(T)\leq(1+o(1))\sqrt{t(t-1)\D(T)}\leq (1+o(1))\sqrt{t(t-1)\D(G)}.
    \]
    This completes the proof.
\end{proof}

\begin{proof}[Proof of Theorem \ref{thm:multipartite}]
    We will begin by proving our lower bound. We observe that for separate parts we are required to use disjoint color sets as two vertices in different parts are adjacent. Let $C_i$ be the set of colors used on the $i$th part $V_i$, where $|V_i|=a_i$. We observe that all pairs of vertices in $V_i$ are of distance $2$ so that no pair of colors appears twice. Each vertex has $\binom{t}{2}$ pairs of colors, so
    \[
    \binom{t}{2}a_i\leq\binom{|C_i|}2\leq \frac{|C_i|^2}2 \qquad \Longrightarrow \qquad \sqrt{t(t-1)a_i}\leq|C_i|.
    \]
    Summing over all $V_i$ yields our lower bound.\\
    
    We now demonstrate our upper bound. We will consider the tree with vertex set  $\{v\} \cup V_i$ and where $v$ is adjacent to every vertex in $V_i$. In other words we have the star $S_{a_i}$. We observe that a $t$-tone coloring of this star provides a way to $t$-tone color the part $V_i$ in the graph $K_{a_1, \ldots, a_b}$. It  follows from theorem \ref{thm:trees} that there exists an $a_0$ such that for $a_i\geq a_0$ we have 
    \[
   \tau_t(K_{a_1, \ldots, a_b}) \le  \sum_{i=1}^b\tau_t(S_{a_i})\leq (1+\delta)\sum_{i=1}^{b}\sqrt{t(t-1)a_i}.
    \]
    This completes our proof.
\end{proof}

\section{2-tone coloring graphs with small chromatic number}\label{sec:smallchi}

In this section we prove Theorems \ref{thm:smallchi1} and \ref{thm:smallchi2}.
\begin{proof}[Proof of Theorem \ref{thm:smallchi1}]
Consider a graph $G$ with $\chi(G)=k$. Consider a proper $k$-coloring of $G$ and let $V_1, \ldots, V_k$ be the color classes. We will give a 2-tone coloring of $G$ using a set of colors $C_1 \cup \ldots \cup C_k$ where the sets $C_i$ are disjoint and all the vertices in $V_i$ get colors from $C_i$. Using this plan means that we will not have to worry about vertices at distance 1 from each other. Indeed, whenever $uv \in E(G)$, $u$ and $v$ must be in different sets say $V_i$ and $V_j$ and so they would not share any colors since $C_i$ and $C_j$ are disjoint. \\

Now consider a pair of vertices $u, v$ with $d_G(u, v)=2$. We need to ensure that $u$ and $v$ do not get the same pair of colors. If $u$ and $v$ are in different sets say $V_i$ and $V_j$ then we are done as before. To address the case where $u, v$ are in the same set $V_i$ we do the following. Let $G_i$ be the graph with vertex set $V_i$ such that for any two vertices $u, v \in V_i$ we have $uv \in E(G_i)$ if $d_G(u, v) =2$. As long as we can assign to each vertex in $V_i$ a pair of colors in such a way that $u$ and $v$ get different pairs of colors whenever they are adjacent in $G_i$ (for all $i$), then we get a valid 2-tone coloring. In other words we need  $\binom{|C_i|}{2} \ge \chi(G_i)$, so it suffices to have $|C_i| = 1+\ceilbrac{\sqrt{2\chi(G_i)}}$. Altogether the number of colors used in our 2-tone coloring is 
\[
\sum_{i=1}^k |C_i| = \sum_{i=1}^k \rbrac{1+\ceilbrac{\sqrt{2\chi(G_i)}}} = k+ \sum_{i=1}^k \ceilbrac{\sqrt{2\chi(G_i)}}.
\]
This completes the proof.
\end{proof}

\begin{proof}[Proof of Theorem \ref{thm:smallchi2}]
   Let $G$ be a graph with maximum degree $\Delta(G)=d$ and chromatic number $\chi(G)=k$. We assume that for every vertex $u$ there are at most $d^4/f$ pairs of vertices $v, w$ such that $d_G(u, v) = d_G(u, w) = d_G(v, w)=2$. Let $V_1, \ldots, V_k$ be the color classes in some proper $k$-coloring of $G$, and let $G_1, \ldots, G_k$ be the graphs described in Theorem \ref{thm:smallchi1}. Then each graph $G_i$ has maximum degree at most $d(d-1) \le d^2$, and our assumptions on $G$ imply that for all $v \in V_i$, $N_{G_i}(v)$ contains at most $d^4/f$ edges of $G_i$. Thus Theorem \ref{thm:AKS} implies that $\chi(G_i) \le cd^2/\log f$ (for the value of $c$ guaranteed by the theorem). Now Theorem \ref{thm:smallchi1} gives us 
   \[
   \tau_2(G) \le k+ \sum_{i=1}^k \ceilbrac{\sqrt{2\chi(G_i)}} \le k \rbrac{1+ \sqrt{2 cd^2 / \log f}}.
   \]
   Letting $C=\sqrt{2c}$ completes the proof. 
\end{proof}

\section{Cartesian powers}\label{sec:powers}

In this section we prove Theorems \ref{thm:cube}, \ref{thm:cliqueproduct} and \ref{TOLS}, as well as Corollary \ref{cor:MOLS}.

\begin{proof}[Proof of Theorem \ref{thm:cube}]
    We aim to use Theorem \ref{thm:smallchi2} to justify this claim. We make note that $d:=\D(G^b)=\Theta(b)$. For a fixed $v\in V(G^b)$ we aim to count the number of $u,w\in V(G^b)$ such that $d(u,v)=d(u,w)=d(v,w)$. We observe a pair of vertices can vary by at most 2 coordinates so to pick our first vertex we have for some constant c, $c\D^2$ ways to choose a vertex $u$ and $w$ cannot deviate from $v$ with a separate pair of vertices from $u$, indeed we have $c'b$ ways to choose our next vertex $w$. Thus we have $\Theta(b^3)=\Theta (d^4/b)$ ways to choose such a pair for fixed $v$. \\

    Thus by Theorem \ref{thm:smallchi2} we have that 
    \[\tau_2(G^b)\leq\chi(G^b)\left(1+\frac{C'b}{\sqrt{\log (b)}}\right)=k\left(C\frac{b}{\sqrt{\log(b)}}\right).\]

   Above, we used the fact that $\chi(G^b) = \chi(G)$. This completes our proof.
\end{proof}

\begin{proof}[Proof of Theorem \ref{thm:cliqueproduct}]

 Fix integers $b, t \ge 1$ and a real number $\d>0$. Let $C$ be a set of $(1+\delta)tn$ colors. We will apply Theorem \ref{thm:DP}.  Let $\mc H$ by the bipartite hypergraph with bipartition $X \cup Y$ where $X=V(K_n^b)=[n]^b$ and where each vertex in $Y$ describes $b-i$ coordinates of a vertex in $K_n^b$ and $i$ colors for some $i$ with $1\leq i \leq t$. More formally, $Y= \bigcup_{ i \in[t]} Y_i$ where \[Y_i = \cbrac{(z_1, \ldots, z_b): z_j \in [n] \mbox{ for exactly } b-i \mbox{ values of } j, \mbox{ and } z_j = * \mbox{ for the other values } j} \times \binom{C}{i}. \]
When $z_j=*$ we treat it as an unspecified coordinate, meaning we view a vertex in $Y$  as describing $b-i$ coordinates of a vertex in $K_n^b$ and $i$ colors. We now describe the edges of $\mc{H}$. For each vertex $v \in V(K_n^b)$ and each set $S\in \binom Ct$, $\mc{H}$ will have an edge $e_{v, S}$ which contains the vertex $v \in X$ and every possible vertex in $Y$ that can be obtained by replacing $i$ coordinates of $v$ with ``$*$'' and taking a subset of $i$ colors from $S$ for some $1 \le i \le t$. We intend for the edge $e_{v, S}$ to correspond to assigning the set of colors $S$ to the vertex $v$. Thus, an $X$-perfect matching in $\mc H$ corresponds to assigning a set of $t$ colors to each vertex in $K_n^b$. In the next paragraph we explain why this assignment is a $t$-tone coloring. \\

The distance between two vertices of $K_n^b$ is exactly their Hamming distance. In other words, if $u,v\in V(K_n^b)$ and $d(u,v)=i$, then there are exactly $b-i$  coordinates where $u$ and $v$ agree. In this case, $u$ and $v$ could have at most $i-1$ colors in common in a $t$-tone coloring. Say we assign the color set $S_u$ to $u$ and $S_v$ to $v$. If our coloring violated the $t$-tone requirement there would be a set $S\subseteq S_u \cap S_v$ of size $i$. But then the edges $e_{u, S_u}$ and $e_{v, S_v}$ would share the vertex in $Y$ obtained as follows: for every coordinate where $u$ and $v$ disagree, replace it with ``$*$'', and use the color set $S$. Thus, $e_{u, S_u}$ and $e_{v, S_v}$ could not both be in a matching of $\mc H$. \\

It remains to show that $\mc H$ has an $X$-perfect matching, and for that we apply Theorem \ref{thm:DP}. 

We let
\[
\eps=\frac{1}{2t}, \qquad d=\frac{(tn)^t(1+\d)^{t-1}}{t!}.
\]

Let us check condition \ref{cond:h1}. The degree of a vertex in $X$ is just the number of possible colorings of a vertex, which is 
\[\binom{tn(1+\d)}{t}\geq\frac{(tn(1+\d)-t)^t}{t!}=\frac{(tn)^t(1+\d-\frac{1}{n})^t}{t!}.
\]
Let $s$ be the expression on the right hand side above. We claim that $s \ge (1+d^{-\alpha})d$ for sufficiently large $n$. Indeed, for fixed $\delta, t$ we have

\[
\lim_{n \rightarrow \infty} \frac{s}{(1 + d^{-\alpha})d}=\lim_{n \rightarrow \infty}\frac{\rbrac{1+\d-\frac{1}{n}}^t}{(1 + d^{-\alpha})\rbrac{1+\d}^{t-1}}= 1 + \delta >1 .
\]
Thus, every vertex in $X$ has degree at least $(1+d^{-\alpha})d$ as required. \\

 The degree of a vertex in $Y_i$ is the number of possible ways to choose a vertex and a set of $t$ colors if $i$ colors and $t-i$ coordinates of the vertex coordinates are fixed. Thus, the degree of a vertex in $Y_i$ is 
 \[
  n^{i}\binom{tn(1+\d)-i}{t-i}\leq n^{i} \frac{((tn)(1+\d))^{t-i}}{(t-i)!}=n^t(1+\d)^{t-i}\frac{t^{t-i}}{(t-i)!}\leq \frac{(tn)^t(1+\d)^{t-1}}{t!}=d,
 \]
 where the last inequality uses that $i \ge 1$. Thus every vertex in $Y$ has degree at most $d$, and so \ref{cond:h1} holds. \\

Let us check \ref{cond:h2}. We fix two vertices in $\mc{H}$ and bound their codegree. We consider cases based on whether our fixed vertices are in $X$ or $Y$. Of course, two vertices in $X$ cannot be in any edge together in a bipartite hypergraph. Next, 
 consider the codegree of a vertex in $X$ and the other in $Y$.
Our fixed vertex in $X$ gives us a vertex in $K_n^b$, and our fixed vertex in $Y$ determines at least one color. Thus the codegree is at most 
\[
\binom{(1+\d)tn-1}{t-1}=O\rbrac{n^{t-1}}=O\rbrac{d^{1-\frac{1}{t}}}\leq d^{1-\frac{1}{2t}}= d^{1-\epsilon}
\]
as required. Finally, consider the codegree of two vertices in $Y$, say one vertex $u \in Y_i$ and one in $v\in Y_j$ with $i \le j$. 
 Assume $k\leq i$ of our coordinates are specified by neither $u$ nor $v$, and that the color sets specified by $u$ and $v$ have $s \leq i$ colors in common. This means that to determine an edge of $\mc{H}$ we have to choose $k$ coordinates for the vertex of $K_n^b$ and $t-i-j+s$ colors. So $k, s \le i \le j$ and also one of the following holds: $i<j$, $k<i$, or $s<i$. Indeed, if we had $i=j=k=s$ then $u$ and $v$ would be the same. The codegree of $u$ and $v$ is then 
 \[
 n^{k}\binom{(1+\d)tn-i-j+s}{t-i-j+s} = O\rbrac{n^{t -i-j+s+k}}= O\rbrac{n^{t-1}}\le d^{1-\frac{1}{2t}}
 \]
 as required. This completes the proof of Theorem \ref{thm:cliqueproduct}.

\end{proof}

\begin{proof}[Proof of Theorem \ref{TOLS}]
Supposing there exists a collection of $n \times n$ MOLS of cardinality $t$, we will color $K_n^2$ as follows. Begin by replacing the entries in each of the $t$ MOLS using disjoint sets of colors. Each of our original Latin squares has entries in $[n]$, so we use $tn$ colors. Suppose our new Latin squares (after replacing the entries with colors) are $L_1, \ldots, L_t$ and denote by $L_i(a, b)$ the color in the $a$th row and $b$th column of $L_i$. Assign the vertex $(a,b)\in V(K_n^2)$ the set of colors $\{L_1(a, b), \ldots, L_t(a, b)\}$ comprised of the color on the $a$th row and $b$th column of each Latin square.\\

We now show this is a $t$-tone coloring. Suppose we have adjacent vertices in $ K_n^2$, say without loss of generality $(a,b)$ and $(a,c)$ where $b\not=c$. We will notice that if they share a color, then as our Latin squares have disjoint color sets, there exists a Latin square $L_i$ in which they were both assigned the same color. But then $L_i(a, b) = L_i(a, c)$ which is impossible since $L_i$ is a Latin square. \\

The only remaining case to check is distance two vertices, as there are no distance 3 vertices in $K_n^2$. Suppose we have two vertices $(a,b), (c,d)$ at distance 2, meaning $a\not=c$ and $b\not=d$. Suppose these two vertices share two colors. This would mean on two Latin squares $L_i, L_j$ where $L_i(a, b) = L_i(c, d)$ and $L_j(a, b) = L_j(c, d)$. This means the Latin squares $L_i, L_j$ were not orthogonal, which is a contradiction. This completes the proof.
\end{proof}

\begin{proof}[Proof of Corollary \ref{cor:MOLS}]
    We let $N(n)$ be the largest possible cardinality of a collection of $n \times n$ MOLS. Thus Theorem \ref{TOLS} says that if $t \le N(n)$ then $\tau_t(K_n^2) = tn$. Now we use some known bounds on $N(n)$.\\

    \ref{case:beth} follows from Beth \cite{MOLS}, who proved that $N(n) \ge n^{5/74}$. \ref{case:macneish} follows from Macneish \cite{McNT}, who proved that if $n=\prod _i^mp_i^{\alpha_i}$ is the prime factorization of $n$, then  $N(n) \ge \min_{i\in[n]}(p_i^{\alpha_i}-1)$. Note that this proves a more general bound than stated in \ref{case:macneish} which is better than \ref{case:beth} in many cases where the prime powers in the factorization of $n$ are all large. 
\end{proof}

\section{$t$-tone coloring a fixed graph $G$ for large $t$}\label{sec:larget}

\begin{proof}[Proof of Theorem \ref{thm:larget}]

    To show that $\tau_t(G)\leq tn- \sum_{u,v\in V(G)}\big(d(u,v)-1\big)$ we describe a $t$-tone coloring using that many colors. We arbitrarily order the vertices $\{v_1,v_2,...,v_n\}$. We then assign sets of colors to our vertices in order. When we color $v_i$ we will use, for each $w \in\{v_1,v_2,...,v_{i-1}\}$, $d(v_i,w)-1$ colors that are currently assigned to $w$ but no vertex besides $w$. Note that since $t \ge (D-1)(n-1)$, the number of colors currently assigned to $w$ and only $w$ is at least $t - (D-1)(i-2) \ge D-1 \ge d(v_i, w)-1$ and so we can always find enough of these colors for $v_i$. To complete a set of $t$ colors for $v_i$, we use $t-\sum_{u\in V(G)}\big(d(v_i,u)-1\big)$ new colors which have never been used yet. This gives a $t$-tone coloring with precisely $tn- \sum_{u,v\in V(G)}\big(d(u,v)-1\big)$ colors. \\

    We will now prove the lower bound $\tau_t(G)\geq tn- \sum_{u,v\in V(G)}\big(d(u,v)-1\big)$ by use of the inclusion-exclusion principle. Consider an arbitrary $t$-tone coloring $f$ of $G$ using color set $C$,  so $f(v)$ denotes the color set assigned to $v$. Now we observe that (explanation follows) 
\begin{align*}
    |C| & \ge \abrac{\bigcup_{v \in V(G)} f(v)} = \sum_{i=1}^n(-1)^{i+1}\sum_{S\in\binom{V(G)}{i}}\left|\bigcap_{u\in S}f(u)\right|\geq\sum_{i=1}^2(-1)^{i+1}\sum_{S\in\binom{V(G)}{i}}\left|\bigcap_{u\in S}f(u)\right|\\
    &=tn-\sum_{u,v\in V(G)}\left|f(u) \cap f(v) \right|\geq tn-\sum_{u,v\in V(G)}(d(u,v)-1).
\end{align*}
Indeed, the equality on the first line is by inclusion-exclusion, and the next inequality is because truncating the alternating sum given by inclusion-exclusion after a negative term gives an underestimate. The inequality on the second line follows from the fact that $\sigma$ is a $t$-tone coloring.
    
\end{proof}

\section{Evidence supporting Conjecture \ref{conj:tree}}\label{sec:treeconj}

In this section we prove some small cases of Conjecture \ref{conj:tree}.

\begin{theorem}\label{ExtTr}
We have

\[   \tau_3(S_\D) = \tau_3(T_\D)=\left\{
\begin{array}{ll}
      8 & \Delta=2 \ \\
      9 & 3\leq \Delta \leq 4 \\
      10 & 5\leq \Delta\leq 7 \\
\end{array} 
\right. \]
\end{theorem}

\begin{proof}
     We consider the following cases.
     \begin{enumerate}

             \item $\Delta=2$
       \\  
     The following was proved by Bickle and  Phillips in \cite{TrTn} where $P_n$ is a path with $n$ vertices (the following is still valid if $n = \infty$):  \\
    \begin{equation}\label{eqn:path}
    \tau_t(P_n)=\sum_{i=0}^{n-1}\max\left\{0,t-\binom i2\right\}.
    \end{equation}
\\
    Since $S_2 = P_3$ and $T_2 = P_\infty$, we have $\tau_3(S_2) = \tau_3(T_2)= 8$ as required.\\

\item $3\le \Delta\le 4$
\\
    It suffices to prove that $\tau_3(S_3) \ge 9$ and $\tau_3(T_4) \le 9$. We first prove $\tau_3(S_3) \ge 9$.  Suppose for the sake of contradiction that we have a 3-tone coloring of $S_3$ using only the colors $1, \ldots, 8$.  We must use 3 colors, say without loss of generality $(123)$, on the head of the star that are not used on the leaves. We can assume there exist $2$ leaves that overlap by a color as otherwise we would need $12$ colors. Thus WLOG assume we have two leaves colored $(456)$ and $(478)$. But now by the pigeonhole principle there is no way to color the final leaf using three of the colors $4, \ldots, 8$ without sharing two colors with another leaf. This completes the proof of the lower bound. \\

    We now prove the upper bound $\tau_3(T_4) \le 9$.  We color $T_4$ as follows. Adopt the same partition as Theorem $\ref{thm:trees}$, namely that for a fixed vertex $v\in T_4$ if $d(u,v)=i$ then $u\in P_i$. We will color $v$ with the color set $(123)$ and $P_1$ with the color sets $(456)(478)(957)(968)$, noting that every pair of vertices in $P_1$ share a color. We continue our coloring inductively. Suppose we have colored $P_0 \cup \cdots \cup P_i$ so that any pair of vertices at distance 2 share a color (we consider this an additional induction hypothesis). 
    
    For each $v\in P_i$, we describe now how to color its three children, say $x, y, z \in P_{i+1}$. Let $u \in P_{i-1}$ be the parent of $v$. Now suppose without loss of generality that $u$ is colored $(456)$ and $v$ is colored $(123)$. We color $x, y, z$ with the sets $(478)$, $(957)$, and $(968)$. We notice that by this method every pair of distance $2$ vertices share exactly one color and no adjacent vertices share a color. Our remaining concern is that one of $x, y, z$ might have been assigned the same color set as some vertex at distance 3. Suppose a vertex $w$ is at distance 3 from $x, y$ or $z$ and $w$ has already been colored. Then $w \in N(u)\setminus \{v\}$ and so $w$ shares a color with $v$. Since $x, y, z$ do not share any colors with $v$, they cannot have the same set of colors as $w$. Thus we are able to color $P_{i+1}$ maintaining the property that all pairs of vertices at distance 2 share a color, completing the proof that $\tau_3(T_4) \le 9$.\\

    \item $5 \le \Delta \le7$ 

    It suffices to show $\tau_3(S_5) \ge 10$ and $\tau_3(T_7) \le 10$. Observation \ref{obs:DeltaLB} gives us $\tau_{3}(S_5) \geq 10$. We proceed to prove $\tau_3(T_7) \le 10$.\\

 We will use the ten colors $0, \ldots, 9$. We emulate the argument from the previous case, i.e.~we color $T_7$ by inductively coloring $P_0$, then $P_1$ and so on, maintaining the property that vertices at distance 2 always share a color. We refer to Figure \ref{fig:graph3}. The top vertex labeled $(123)$ is the root of the tree and the figure shows how to color its seven children. Now assume we have colored $P_0, \ldots, P_i$, and we now want to color the six children of some $v \in P_i$ whose parent is $u \in P_{i-1}$. Without loss of generality, say  we have colored $v$ the color set $(689)$ and $u$ with $(123)$.  Then we can color the six children of $v$ with the sets $(247), (345), (157), (205), (307), (401)$.\\

 We observe that this coloring method can be drawn from the unique $3$ uniform hypergraph on $7$ vertices such that every pair of vertices is represented in a unique edge. This hypergraph is known as Fanos plane. We assign the colors of $v$ to no vertex and make the colors assigned to $u$ represent $3$ vertices that make up an edge and arbitrarily complete the labeling of Fanos plane from the colors not yet utilized from $\{0,1,\ldots,9\}$.

    \end{enumerate}

\end{proof}

\begin{figure}
    \centering
    \begin{tikzpicture}
        \node[draw, circle] (789) at (0, 0) {123};
        \node[draw, circle] (012) at (-6, -2) {456};
        \node[draw, circle] (234) at (-4, -2) {478};
        \node[draw, circle] (450) at (-2, -2) {579};
        \node[draw, circle] (135) at (0, -2) {689};
        \node[draw, circle] (265) at (2, -2) {058};
        \node[draw, circle] (164) at (4, -2) {067};
        \node[draw, circle] (063) at (6, -2) {049};

        \node[draw, circle] (b) at (-5, -4) {247};
        \node[draw, circle] (c) at (-3, -4) {345};
        \node[draw, circle] (d) at (-1, -4) {157};
        \node[draw, circle] (e) at (1, -4) {205};
        \node[draw, circle] (f) at (3, -4) {307};
        \node[draw, circle] (g) at (5, -4) {401};

        \draw (789) -- (012);
        \draw (789) -- (234);
        \draw (789) -- (450);
        \draw (789) -- (135); 
        \draw (789) -- (265); 
        \draw (789) -- (164); 
        \draw (789) -- (063); 

        \draw (135) --  (b);
        \draw (135) --  (c);
        \draw (135) --  (d);
        \draw (135) --  (e);
        \draw (135) --  (f);
        \draw (135) --  (g);
        
    \end{tikzpicture}
    \caption{How to 3-tone color $T_7$}
    \label{fig:graph3}
\end{figure}
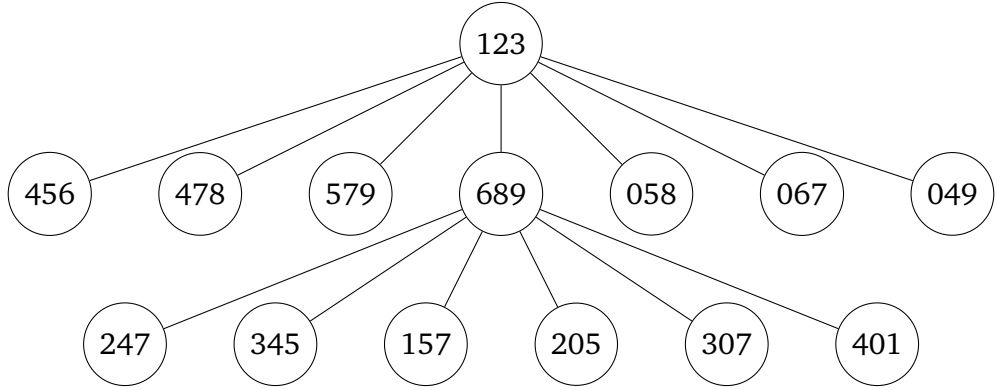

\begin{theorem}
    We have
\[   \tau_4(S_\D) = \tau_4(T_\D)=\left\{
\begin{array}{ll}
      13 & \Delta=3  \\
      14 & \Delta=4 \\
\end{array} 
\right. \]
For $\Delta=2$ we have $\tau_4(S_2)=11 < \tau_4(T_2)=12$. However for any tree $T \neq S_2$  with $\Delta(T)=2$ we have $\tau_4(T)=12$.
\end{theorem}

\begin{proof}
    
We consider the following cases. 
\begin{enumerate}
    \item $\Delta=2$ 
\\
    Suppose $\Delta(T)=2$, so $T$ is a path. We use \eqref{eqn:path}. If $T\neq S_2$ then $T$ has  $n \ge 4$ vertices and we get $\tau_4(T)=12$. If $T=S_2$ then $n=3$ and we get $\tau_4(T)=11$.\\
    
    \item \label{case:Delta=3} $\Delta=3$.
    
    It suffices to show $\tau_4(S_3) \ge 13$ and $\tau_4(T_3) \le 13$. In any 4-tone coloring of $S_3$, the head vertex gets 4 colors which cannot be used on the leaves. Any pair of leaves can share at most one color, so we need at least $4+4+(4-1)+(4-2)=13$ colors. Thus $\tau_4(S_3) \ge 13$.\\

    We move on to proving $\tau_4(T_3) \le 13$. We use the color set $\{0, \ldots, 9, \alpha, \beta, \gamma\}$. We will prove by induction that we can color $T_3$ such that the following conditions are satisfied:
    \begin{enumerate}
        \item \label{cond:d1} adjacent vertices share no colors,
        \item \label{cond:d2} for any vertex $x$, every pair of vertices in $N(x)$ shares exactly one color, but no color is shared by all three vertices in $N(x)$,
        \item \label{cond:d3}every pair of vertices at distance 3 shares exactly 2 colors,
        \item \label{cond:d4}vertices at distance 4 do not get the same color set.
    \end{enumerate}
    Note that these conditions imply we have a $4$-tone coloring.\\

    We begin by assigning the root $v$ the color set $(1234)$ and the vertices of $u_1,u_2,u_3\in P_1$ the colors $(5678), (590\alpha), (69\beta\gamma)$ respectively. We will color $P_2$ as follows
    \begin{align}
        \sigma\left(N(u_1)\cap P_2\right)=&\{(190\beta),(29\alpha\gamma)\}\label{line:10}
        \\
        \sigma\left(N(u_2)\cap P_2\right)=&\{(167\beta),(268\g)\}\label{line:11}
        \\
        \sigma\left(N(u_3)\cap P_2\right)=&\{(1570),(258\alpha)\}\label{line:12}
    \end{align}

    We observe that this is a proper partial $4$-tone coloring satisfying $(a)$ through $(d)$. 
    $(a)$ is easy to verify. 
    We observe that $(b)$ only needs to be checked for $x \in \{v,u_1,u_2,u_3\}$. The case $x=v$ follows from the way we colored $P_1$, and the other cases follow from the coloring of $P_2$. In particular, every vertex in $P_2$ shares exactly one color with $v$, and one can check from the coloring given on lines \eqref{line:10}--\eqref{line:12} that each pair of vertices at distance 2 shares a color (the second color, e.g. the color 9 on line \eqref{line:10}). Thus $(b)$ holds.
   
    Now we check $(c)$. For a pair of distance $3$ vertices which have been colored so far, one of them must be some $u_i \in P_1$ and the other must be a vertex from $P_2-N(u_i)$, say $x \in N(u_j)$ where $j \neq i$. 
    It is routine to check that $(c)$ holds in all cases here. For example, $u_1$ is colored $(5678)$ and shares exactly two colors with all the vertices from lines \eqref{line:11} and \eqref{line:12}. Thus $(c)$ holds.
    The only pairs of distance $4$ vertices are in $P_2$ and none of these have the same color set so $(d)$ holds.\\
    
    We now wish to inductively color our vertices as follows, suppose all of $P_i$ and all previous layers have been colored. Our induction hypothesis is that conditions $(a)$ through $(d)$ are satisfied in our coloring.\\
    
    We will color the vertices in $P_{i+1}$ in groups where all the vertices in each group have the same grandparent in $P_{i-1}$. For a pair of vertices $u,v\in P_{i-1}$ we observe that they have no common grandchildren (i.e. $N(N(v))\cap  N(N(u)) \cap P_{i+1}=\varnothing$), so any grandchild of $u$ is distance more than 4 away from any grandchild of $v$. Thus there is no possible problem that could arise from a pair of vertices consisting of a grandchild of $u$ and a grandchild of $v$.\\

     Now consider a $v\in P_{i-1}$, and we will color its grandchildren. Let $u_1,u_2,u_3$ be the vertices in $N(v)$ where $u_3\in P_{i-2}$ and $u_1, u_2 \in P_i$. Suppose WLOG that $v$ is colored $(1234)$. We recall that by $(b)$ that $u_k$ and $u_j$ share a unique color, call this color $c_{k, j}$. We also know that $u_k$ has two colors that do not appear on any vertex at distance 2, and call these colors $c'_{k, 1}$ and $c'_{k, 2}$.
     
 We color $N(N(v))\cap P_{i+1}$ as follows:
        \begin{align}
        \sigma\left(N(u_1)\cap P_{i+1}\right)=&\{(1c'_{2, 1}c'_{3, 1}c_{2,3}),(2c'_{2, 2}c'_{3, 2}c_{2,3})\}\label{line:13}
        \\
        \sigma\left(N(u_2)\cap P_{i+1}\right)=&\{(1c'_{1, 1}c'_{3, 1}c_{1,3}), 
    (2c'_{1, 2}c'_{3, 2}c_{1,3})\}\label{line:14}
    \end{align}

    We now check conditions $(a)$--$(d)$. 
    \begin{figure}
    \centering
    \begin{tikzpicture}

        \node[draw] (689) at (0, 2) {$c_{1,3}c_{2,3}c'_{3,1}c'_{3,2}$};
        
        \node[draw] (123) at (0, 0) {1234};
        
        \node[draw] (067) at (2, -2) {$c_{1,2}c_{1,3}c'_{1,1}c'_{1,2}$};
       
        \node[draw] (478) at (-2, -2) {$c_{1,2}c_{2,3}c'_{2,1}c'_{2,2}$};

        \node[draw] (345) at (-4, -4) {$1c'_{2, 1}c'_{3, 1}c_{2,3}$};
        \node[draw] (307) at (4, -4) {$2c'_{1, 2}c'_{3, 2}c_{1,3}$};
        \node[draw] (5) at (-1.5, -4) {$2c'_{2, 2}c'_{3, 2}c_{2,3}$};
        \node[draw] (7) at (1.5, -4) {$1c'_{1, 1}c'_{3, 1}c_{1,3}$};
        
        \draw (689) --  (123);
        \draw (478) --  (123);
        \draw (067) --  (123);
        \draw (067) --  (307);
        \draw (067) --  (7);
        \draw (478) --  (5);
        \draw (478) --  (345);

    \end{tikzpicture}
    \caption{How to 4-tone color $T_3$}
    \label{fig:graph4}
\end{figure}
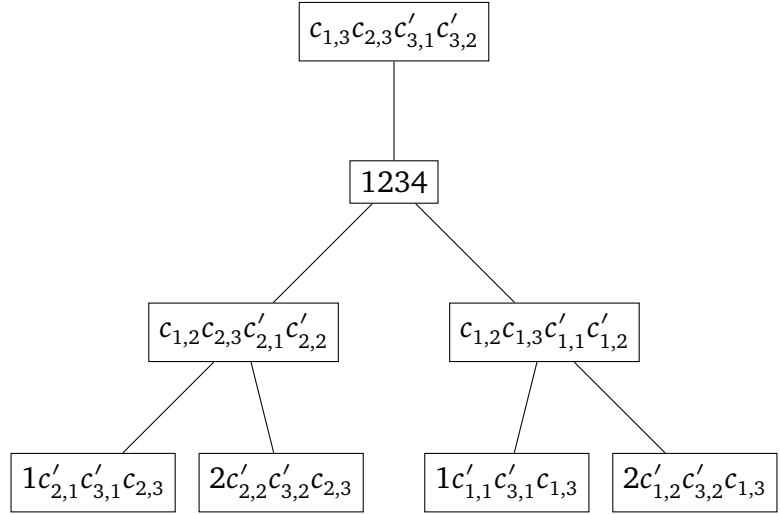
It is easy to see that $(a)$ holds since no pair of adjacent vertices share a color. We check $(b)$ next. The only vertices $x$ to check are $u_1, u_2$ since they are the only vertices that have newly colored neighbors. To check the case $x=u_1$, we note that the three vertices in $N(u_1)$ are colored $(1234),(1c'_{2, 1}c'_{3, 1}c_{2,3})$, $ (2c'_{2, 2}c'_{3, 2}c_{2,3})$ which satisfies $(b)$ by inspection (in particular there is no color in all three sets since $c_{2,3} \notin \{1, 2, 3, 4\}$). Similarly the case $x=u_2$ satisfies $(b)$. \\

    We check $(c)$. Consider a pair of distance $3$ vertices, and observe that any colored vertex at distance $3$ from a newly colored vertex is in $N(v)=\{u_1, u_2, u_3\}$. By inspection, $u_3$ shares two colors with each of the sets on lines \eqref{line:13} and \eqref{line:14}. Also, $u_1$ shares two colors with each set on line \eqref{line:14}, and $u_2$ shares two colors with the sets on line \eqref{line:13}. Thus $(c)$ holds.

    We now check $(d)$. But this easily follows from $(a)$ and $(c)$. Indeed, consider any pair of colored vertices $x, y$ at distance 4, and note that by our coloring procedure (in particular, coloring all of $P_0, \ldots, P_i$ before $P_{i+1}$) we have already colored every vertex in the path from $x$ to $y$, say $P=(x=x_1,x_2,x_3,x_4,x_5=y)$. Now by $(a)$, $x=x_1$ shares no color with $x_2$, but by $(c)$ $x_2$ shares two colors with $x_5=y$. Therefore $x$ cannot have the same color set as $y$, verifying $(d)$.\\

    \item $\Delta=4$
    \\
    By similar logic to our lower bound in the last case \eqref{case:Delta=3}, to 4-tone color the star $S_4$ we need at least $4+4+(4-1)+(4-2)+(4-3)=14$ colors.\\ 
    
We will now color the infinite tree $T_4$ using the colors $\{0, \ldots, 9, \alpha, \beta, \gamma, \delta\}$. Our coloring will have the following properties:
    \begin{enumerate}
        \item \label{cond:d1'} adjacent vertices share no colors,
        \item \label{cond:d2'} for any vertex $x$, every pair of vertices in $N(x)$ shares exactly one color, but no color is shared by three vertices in $N(x)$,
        \item \label{cond:d3'}every pair of vertices at distance 3 shares either one or two colors,
        \item \label{cond:d4'}vertices at distance 4 do not get the same color set.
    \end{enumerate}
We will assign the root vertex the color set $(1234)$ and the vertices $u_1, u_2, u_3, u_4$ of $P_1$ the color sets $(5678),(590\alpha)$, $(69\beta\gamma),(70\beta\delta)$ respectively. We color $P_2$ as follows:
\begin{align}
        \sigma\left(N(u_1)\cap P_2\right)=&\{(29\beta\d),(3\beta0\alpha),(409\g)\}\label{line:15}
        \\
       \sigma\left(N(u_2)\cap P_2\right)=&\{(3\beta78),(476\gamma),(16\beta\d)\}\label{line:16}
        \\
        \sigma\left(N(u_3)\cap P_2\right)=&\{(475\alpha),(150\d),(2078)\}\label{line:17}
        \\
        \sigma\left(N(u_4)\cap P_2\right)=&\{(159\g),(2968),(365\alpha)\}\label{line:18}
    \end{align}
We argue that the coloring given so far satisfies the conditions $(a)$--$(d)$. We can easily verify adjacent vertices share no colors, so $(a)$ holds. We check $(b)$ next. We observe that for each pair $u_i, u_j$ that they share exactly one color. Additionally each vertex in $P_2$ is distance $2$ from $v$ and only the first color (in the sets given on lines \eqref{line:15}--\eqref{line:18}) is shared with $v$. Also, any pair of color sets on line \eqref{line:15} share exactly one color, and the same can be said of lines \eqref{line:16}--\eqref{line:18}. Thus $(b)$ holds. To check $(c)$, there are several cases. For example, the reader can verify that the color set $(5678)$ for $u_1$ shares either one or two colors with each set on lines \eqref{line:16}--\eqref{line:18}. The rest of the cases for $(c)$ involve checking that the color sets given to $u_2, u_3, u_4$ share either one or two colors with the appropriate lines in \eqref{line:15}--\eqref{line:18}.\\

    Suppose we have colored $P_0, \ldots, P_i$ for some $i \ge 2$. Now we wish to color $N(N(v))\cap P_{i+1}$ for each $v \in P_{i-1}$. Let $N(v)=\{u_1, u_2, u_3, u_4\}$ where $u_4\in P_{i-2}$ and $u_1, u_2, u_3 \in P_i$. Denote by $c_{s, j}$ the unique color shared by $u_s$ and $u_j$. Denote by $c'_k$ the color $u_k$ shares with no other vertex in $N(v)$. Thus, each vertex $u_s$ has the color set 
    \[
    \big( c_{s, j}c_{s, k}c_{s, \ell} c'_s\big)
    \]
    where $\{j, k, \ell \}=[4]\setminus \{s\}$.  We will suppose WLOG that $v$ is colored (1234). We take subscripts modulo 4 so $c'_{s+k}$ and $c_{s+k,s+r}$ are understood.
    For $k=1, 2, 3$ we assign $N(u_k)\cap P_{i+1}$  the following color sets:
\begin{equation}\label{line:19}
    \big((k+1)c_{k+1,k+2}c_{k+2,k+3}c'_{k+3}\big), \;\; 
    \big((k+2)c_{k+3,k+2}c_{k+1,k+3}c'_{k+1}\big), \;\; 
    \big((k+3)c_{k+3,k+1}c_{k+1,k+2}c'_{k+2}\big),
\end{equation}
    where the colors $(k+1), (k+2), (k+3)$ are also taken modulo 4, so for example if $k=3$ then $(k+3)$ denotes the color 2. \\

    We check the properties $(a)$--$(d)$. $(a)$ is easy to check. For $(b)$, observe that each pair of sets on line \eqref{line:19} shares a unique color, and each set share a unique color (its first color) with $v$. \\

     Now we check (c). For a newly colored vertex say  $x \in N(u_k) \cap P_{i+1}$, the vertices of distance $3$  will be the $u_j$ for $j \neq k$. 
     
     Of the last $3$ colors listed in our coloring for $x$ either one or two will overlap with $u_j$ only for $j>1$  and our first listed color will only ever overlap with $u_j$ for $j=1$. Thus $(c)$ is satisfied.\\

    As before (in the $\Delta=3$ case), $(d)$ follows from $(a)$ and $(c)$. This completes our inductive proof that we can color $T_4$. Thus $\tau_4(T_4)= 14$.
    
\end{enumerate}

\end{proof}

We might begin to suspect that in general we can determine the $t$-tone chromatic number of an arbitrary tree by it's maximum degree alone, however this is not the case which we demonstrate below.

\begin{prop}
    $\tau_5(S_3) < \tau_5(T_3)$
\end{prop}

\begin{proof}
It suffices to show there exists a finite tree, $T$, such that $\D(T)=3$ satisfying $\tau_5(S_3)<\tau_5(T)$.
By Theorem \ref{thm:larget} we have $\tau_5(S_3)  =17$. Now consider the following tree $T$, formed by taking $S_3$ and for one leaf $v$ we add 2 vertices adjacent to it. We will prove $\tau_5(T) > 17$. Assume the head of $S_3 \subseteq T$ is colored $D=(\gamma\delta\eps\phi\theta)$ and the leaves are colored $A=(12345),B=(16789),C=(26\alpha\beta\gamma)$. It is easy to see that the coloring of $S_3$ on $17$ colors is unique up to permuting colors, so we can assume we maintain the same coloring and let the vertex we join $2$ vertices to have the coloring $A$. 
Each of the two vertices in $T-S_3$ can have at most one of color from $D$ and $2$ vertices from $B$ and $C$. If such a vertex does not take a color from $D$ we have that by pigeonhole we must take $3$ colors from either $B$ or $C$ which as they are of distance $3$ does not permit a 5-tone coloring. Thus, each vertex in $T-S_3$ must take a color from $D$. Now suppose the two vertices in $T-S_3$ are both assigned the same color from $D$, WLOG call this color $\gamma$. Then these two vertices cannot share an additional color since they have distance 2, but we must fill in the remaining 4 colors on each vertex and we only have $7$ colors to choose from so this cannot occur. Finally, consider the case where the two vertices of $T-S_3$ have distinct colors from $D$. If either of these vertices are assigned the color $6$ they cannot have $2$ other colors from $B$ and they cannot have two other colors from $C$ but as there would be $3$ colors to fill in by pigeonhole we must use $2$ from one of these sets. Thus we know in this instance we do not use the color $6$, but that means we have $4$ colors still needed to be assigned for each vertex which forces these two vertices share at least two colors, a contradiction. Thus $\tau_5(T)>17$.
    \end{proof}

In fact this can be generalized to the following.

\begin{prop}
    For every positive integer $k$ there exists a natural number integer $t$ and a tree $T$ satisfying $\Delta(T)=k$ and $\tau_t(S_k)<\tau_t(T)$
\end{prop}

\begin{proof}
    
    Assuming $t\ge k$ we have by Theorem \ref{thm:larget} that
    \[\tau_t(S_k)=(k+1)t-\binom{k}{2}.
    \]
    Let $T$ be the tree built by taking $S_k$ and connecting one leaf to a new vertex, $v$. We now color $v$ by noting that it is distance $3$ from each leaf meaning it can share 2 vertices from each of these vertices and it can share one vertex with what was the head of the star. Thus we have that if 
    \[
    t-2(k-1)-1=t-2k-3\ge1
    \]
    Then we would need to use at least one new color meaning that if $t\ge 2k+4$ we have that $\tau_t(T)>\tau_t(S_k)$ 
    
\end{proof}

This demonstrates that the issue we ran into prior was not unique to $S_3$. We do notice that the reason this became an issue in the examples shown is that we haven't used enough of our pairings before introducing another vertex.

\section{Concluding remarks} \label{sec:conclusion}

In this note we gave several new results for $t$-tone coloring, but there are plenty of questions remaining. Of course it would be nice to resolve Conjecture \ref{conj:tree}. We also make the following conjecture which would generalize Corollary \ref{cor:MOLS} \ref{case:beth}.
\begin{conjecture}
    For all positive integers $t$ and $b$ there exists an $N=N(t, b)$ such that for every $n\geq N$ we have that $\tau_t(K_n^b)=tn$.
\end{conjecture}

Two of our proofs (for Theorems \ref{thm:trees} and \ref{thm:cliqueproduct}) both encode the $t$-tone coloring problem as a hypergraph matching problem. It would be interesting to find more applications of hypergraph matching results to $t$-tone coloring (and similar coloring problems). 

\bibliographystyle{abbrv}

\end{document}